\documentclass[reqno,11pt,letterpaper]{amsart}
\usepackage[mathscr]{euscript}
\usepackage{amsmath,amsthm,amsfonts,amssymb, amscd, bm}
\usepackage{extpfeil}

\usepackage{tikz}
\usetikzlibrary{matrix}%arrows,automata,shapes.symbols,patterns,calc}
\usepackage[bookmarks=false,colorlinks,linkcolor=mylinkcolor,citecolor=mycitecolor]{hyperref}
	\definecolor{mycitecolor}{rgb}{.1,.5,.1}
	\definecolor{mylinkcolor}{rgb}{.1,.1,.6} %{.6,.1,.2}
	\definecolor{myblue}{rgb}{.3,.3,.8}
	
	\definecolor{mitja}{rgb}{.8,.1,.1}
	\definecolor{aaron}{rgb}{.5,.1,.6}

\usepackage[margin=1.50in]{geometry}

\numberwithin{equation}{section}

\theoremstyle{plain}
  \newtheorem{lemm}[equation]{Lemma}
  \newtheorem{theo}[equation]{Theorem}
  \newtheorem{prop}[equation]{Proposition}

\theoremstyle{definition}
  
  \newtheorem{exam}{Example}
  \newenvironment{xam}[1][\theexam]{\begin{exam}[{#1,} continued]}{\end{exam}}
  \newtheorem{ques}{Question}

\theoremstyle{remark}
  \newtheorem*{rema}{Remark}

\newcommand{\demph}[1]{\emph{\color{myblue}{#1}}}

\def\bs{\boldsymbol}
\def\qand{\quad\hbox{and}\quad}

\def\onto{\twoheadrightarrow}

\def\measures{\mathrel{{-}\hskip-1.23ex\raisebox{.14em}{\tiny$\scriptscriptstyle\backslash\!\!\backslash$}\hskip-1.43ex\twoheadrightarrow}}
\def\pmeasures{\mathrel{{-}\hskip-1.23ex\raisebox{.14em}{\tiny$\scriptscriptstyle\backslash\!\!\backslash$}\hskip-1.43ex\rightarrow}}
\def\covers{\measures}
\def\pcovers{\pmeasures}

\def\p{\varphi}
\def\pinv{\overline\varphi}

\def\shuf{\mathop{\sqcup\!\sqcup}}

\def\arg{\bs\cdot}

\def\ngens{\mathcal N}
\def\sym{\mathrm{Sym}}
\def\qsym{\mathrm{QSym}}
\def\nsym{\mathrm{NSym}}
\def\ncsym{\bs{\Pi}}

\def\omp{\mathrm{OMP}}

\DeclareMathOperator{\Alg}{\mathrm{Alg}}

\def\malg{\mathbf{Malg}}
\def\pcbialg{\ensuremath \mathbf{PCbialg}}
\def\cbialg{\ensuremath \mathbf{Cbialg}}
\def\combinat{\ensuremath \mathbf{Combinat}}
\DeclareMathOperator{\Meas}{\mathsf{Meas}}
\DeclareMathOperator{\Cov}{\mathsf{Cov}}
\DeclareMathOperator{\PCov}{\mathsf{PCov}}

\DeclareMathOperator{\coalg}{\mathsf{coalg}}

\DeclareMathOperator{\can}{\mathsf{can}}
\renewcommand{\deg}{\operatorname{\mathsf{deg}}}

\def\N{\mathbb{N}}
\def\Z{\mathbb{Z}}

\def\ep{\varepsilon}
\def\ot{\otimes}
\def\De{\Delta}

\def\id{\operatorname{id}}
\def\Hom{\operatorname{Hom}}

\def\Zb2{\widehat{\operatorname{Z}}^2_2}

\def\m{\operatorname{m}}

\def\C{\mathcal{C}}

\def\k{\Bbbk}

\author{Aaron Lauve}
\address[Lauve]{
	Department of Mathematics \& Statistics\\
        Loyola University Chicago\\
        Chicago, IL\, 60660 
        (USA)
        }
\email{lauve@math.luc.edu}
\thanks{Lauve partially supported by NSA grant H98230-12-1-0286.}

\author{Mitja Mastnak}
\address[Mastnak]{
	Department of Mathematics \& Computing Science\\
        Saint Mary's University\\
        Halifax, Nova Scotia\, B3H 3C3 
        (CANADA)
	}
\email{mmastnak@cs.smu.ca}
\thanks{Mastnak partially supported by NSERC Discovery Grant 371994-2014}

\title[Coverings]{Bialgebra Coverings and Transfer of Structure}

\begin{document}

\begin{abstract}
We introduce the bicategory of bialgebras with coverings (which can be thought of as coalgebra-indexed families of morphisms), and provide a motivating application to the transfer of formulas for primitives and antipode. Additionally, we study properties of this bicategory and various sub-bicategories, and describe some universal constructions. Finally, we generalize Nichols' result on bialgebra quotients of Hopf algebra, which gives conditions on when the resulting bialgebra quotient is a Hopf algebra.
\end{abstract}

\maketitle

%%------------------------------------------------
% TABLE OF CONTENTS
%\setcounter{tocdepth}{1}% or 2 or 3
%{ \parskip=0pt\footnotesize \tableofcontents}
%%------------------------------------------------

%------------------------------------------------
% Begin SECTION
%------------------------------------------------
\section*{Introduction}\label{sec: intro}
We define and study the bicategory of bialgebras with coverings. Its construction was,
in large part, inspired by an idea due to Grunefelder and Par\'{e} \cite{GP:87} of indexing families of algebra morphisms by coalgebras.
In this note we focus on the algebraic aspects of the theory. Combinatorial applications will be explored further in \cite{LM:xx}. 

The study of this bicategory was originally motivated by applications to combinatorial properties of $\mathbb{N}$-graded connected bialgebras, but we hope that this work will also lead to further developments to the general theory of Hopf algebras.   A case in point is a natural generalization of Nichols' result \cite{Nic:78}, which states that a bialgebra quotient of a Hopf algebra that is either finite dimensional or cocommutative is automatically a Hopf algebra (see Theorems \ref{th: fin-dim-cover} and \ref{th: pointed-cover}).  It should also serve as a nice new example of a bicategory (and also a double category; see the remark of Section \ref{sec: bialg}). 

This work began as an attempt to understand the primitives and antipode of the Hopf algebra $\ncsym$ of symmetric functions in noncommuting variables \cite{LM:11}.   The formulas found there for expressing primitives and the antipode in terms of a distinguished generating set of $\ncsym$ point to the Hopf algebra $\omp$ of ordered multiset partitions (see Example \ref{ex: omp over nsym}). The Hopf algebra $\omp$ is the free algebra generated by all finite subsets of natural numbers and has comultiplication given by partitioning a set in all possible ways.  Even though $\omp$ is, in some sense, much bigger then $\ncsym$, it is much easier to get a handle on the formulas for computing its primitives and the antipode in terms of its generators (finite subsets of $\mathbb{N}$).  
%Even though $\ncsym$ in not a quotient of $\omp$, 
There is a nice family of bialgebra homomorphisms from $\omp$ to $\ncsym$ that jointly covers it; and this family can then be used to transport the structure formulas from $\omp$ to $\ncsym$.

This can be formalized as follows: if $B$ and $A$ are bialgebras and $C$ a coalgebra, then we say that $f\colon B\otimes C\to A$ is a partial covering if it is a measuring ({\it i.e.}, it corresponds to an algebra map from $B$ to the convolution algebra $\Hom(C,A)$) as well as a coalgebra map.  We say that $f$ is a covering if it is also surjective.   The transfer of structure discussed above can then be formalized as follows (cf. Theorem \ref{th: transfer}):

\smallskip\noindent
\textsl{Let ${f} \colon B\ot C \to A$ be a covering. Let $\iota \colon A\to B\ot C$ be any linear section of ${f}$, that is, ${f}\circ \iota = \id$.
Then the following hold.
\begin{enumerate}
\item If $p\in B$ is primitive then for any $c\in C$, ${f}(p,c)$ is primitive in $A$.
\item If $A$ and $B$ are Hopf algebras, then their antipodes are related by the formula $S_A = {f}\circ (S_B\ot\id)\circ \iota$.
\end{enumerate}}

Given two partial coverings $f\colon B\otimes C_f\to A$ and $g\colon B\otimes C_g\to A$, we may also consider morphisms between them. 
These are given by coalgebra maps $\mathbf t\colon C_f\to C_g$ such that $f=g\circ(\id\otimes\mathbf t)$ and lead naturally to the bicategory of bialgebras with partial coverings. This bicategory context allows us to define a concept of a covering-equivalence of bialgebras and Hopf algebras.  We hope that this will serve as a nice organizing principle in classification efforts for various classes of Hopf algebras (in particular, but not limited to, combinatorial Hopf algebras).

This paper is split into four sections:  

In Section \ref{sec: measurings} we introduce several motivating examples and the bicategory of algebras with measurings ({\it i.e.}, we briefly forget about the coalgebra structure on $B$ and $A$ and we merely demand that $f\colon B\otimes C\to A$ is a measuring).  
Again, this follows the seminal work of Grunenfelder and Par\'{e} \cite{GP:87}; though here we realize a bicategory structure by making $C$ a variable (in \cite{GP:87} the choice of $C$ was fixed).  

In Section \ref{sec: coverings} we define the bicategories of bialgebras with coverings and partial coverings and give additional motivating examples. We adapt Sweedler's construction of the universal measuring coalgebra to construct a universal partial covering coalgebra. Additionally, we define and discuss the notion of equivalent coverings. We conclude the section by proving the theorem on transfer of structure discussed above. 

In Sections \ref{sec: subcategories} and \ref{sec: Hopf}, we discuss properties of subcategories of interest. 
In the former, we focus on cocommutative and graded coverings. We also identify weak initial objects for coverings of graded connected cocommutative bialgebras. In Section \ref{sec: Hopf}, we focus on coverings by Hopf algebras. We prove the following: if a Hopf algebra $B$ covers a bialgebra $A$ and if either $A$ is finite dimensional or $B$ is pointed, then $A$ is a Hopf algebra. This is the above mentioned generalization of Nichols' result.
 
\subsubsection*{Notation} Throughout, we suppress the unit map $u \colon \k \to A$ for algebras, identifying $1_\k$ with $1_A$ (and, e.g., writing $u\circ\ep(a)$ as $\ep(a)$ for bialgebra elements $a$).

\subsubsection*{Acknowledgements} We would like to thank the referee for a number of useful comments, in particular for pointing us to the references \cite{HFV:17}, \cite{PR:16} and for explaining a more categorical way of approaching some of the constructions discussed in the paper: see the discussion at the end of Section \ref{sec: measurings}, the paragraph preceding Theorem \ref{th: universal-coalg}, the remark following Proposition \ref{th: universal-bialg}, and the comment preceding Question \ref{q:cocommutative}.

%------------------------------------------------
% Begin SECTION
%------------------------------------------------
\section{Measurings}\label{sec: measurings}

If $A$ and $B$ are algebras and $C$ a coalgebra, then we say that a map $f\colon B\ot C\to A$ is a \demph{measuring} (or measures $B$ to $A$) if it corresponds to a unital algebra map from $B$ to the convolution algebra $\Hom_k(C,A)$. More explicitly, $f$ is a measuring if, for all $b,b'\in B$ and $c\in C$, it satisfies
\begin{eqnarray}
\label{eq: meas-prod} f(bb'\otimes c) &=& \sum_{(c)} f(b\otimes c_1)f(b'\otimes c_2),  \\
\label{eq: meas-counit} f(1\otimes c) &=& \ep(c).
\end{eqnarray}
Measurings were first studied by Sweedler \cite{Swe:69}. 

%------------------------------------------------
\subsection{Motivating examples}\label{sec: examples}
It was Grunenfelder and Par\'{e} \cite{GP:87} who first observed that measurings can be viewed as $C$-indexed families of morphisms from $B$ to $A$. From this perspective, it is natural to abuse notation and write $f(b,c)$ for $f(b\otimes c)$ and we do so freely in what follows. We are most interested in the cases where the family $\left\{ f(\arg, c)\right\}_{c\in C}$ jointly spans $A$, and the most interesting of these come when no single $f(\arg,c)$ is surjective. We now illustrate with several examples (and one non example). 

%\addtocounter{exam}{-1}
\begin{exam}\label{ex: A.id over A}
Let $A,B$ be algebras over $\k$, and $f\colon B\to A$ an algebra map. Give $\k$ the coalgebra structure induced by letting $1$ be group-like. Then $\overline{f}\colon B\ot \k \to A$ given by $\overline f(b,1) = f(a)$ is a measuring. Putting $f=\id$, deduce that every algebra has at least one surjective measuring, albiet not a particularly interesting one.
\end{exam}

\begin{exam}\label{ex: comm over noncomm}
Identify the group algebra of $\Z$ with $\k[z,z^{-1}]$ and let $A$ be a Hopf algebra. Now let $C=\coalg(A)$.  We define a measuring $f\colon \k[z,z^{-1}]\ot C\to A$ by putting $f(z,\arg)=\id$, $f(z^{-1},\arg)=S$, and $f(z^n,\arg)=\id^{* n}$, where $\id^{* n}$ stands for the $n$-th convolution power of $\id\in\Hom_{\k}(A,A)$. In more detail, we have that 
$$
f(z^n,a) = \begin{cases} a_1\cdots a_n & n>0\\
\mu\ep(a) & n=0 \\
S(a_1)\cdots S(a_{-n}) & n<0\end{cases}. 
$$
\end{exam}

Already we catch a glimpse of the possibilities beyond ordinary morphisms: if the algebra $A$ above is noncommutative, then we find it is jointly spanned by mappings from a commutative one. We may ask under what conditions is the $f$ above a coalgebra map. This happens if and only if $A$ is cocommutative. 
(Assuming $A$ is cocommutative, checking the coalgebra map condition is straightforward; the reverse implication follows from Lemma \ref{lem: Ccoc}, noting that $\k[z,z^{-1}]$ is Hopf.)

Above and below, we employ a useful method of building measurings. If a linear map $f(\arg,c)$ is defined for generators of $B$, then we may \emph{decree} $f$ to satisfy \eqref{eq: meas-prod} and \eqref{eq: meas-counit}. Checking that this respects the relations among those generators will frequently be left to the reader.

\begin{exam}\label{ex: nsym over sym}
Let $\k[p_1,p_2,\ldots]$ and $\k\langle H_1,H_2,\ldots\rangle$ denote, respectively, the free commutative and noncommutative algebras on countable generators, one in each degree. These are presentations of what are known in algebraic combinatorics as the rings of \emph{symmetric functions} ($\sym$) and \emph{noncommutative symmetric functions} ($\nsym$), respectively. Recall that for any positive integer $n$, we say $\lambda$ is a partition of $n$ (written $\lambda \vdash n$), if $\lambda$ is an tuple of positive integers $(\lambda_1\geq \ldots \geq \lambda_r)$ summing to $n$ (written $|\lambda| = \lambda_1+\cdots+\lambda_r = n$). We take $(0)$ to be the unique partition of $0$. Then $\sym$ has basis $\{p_\lambda:=p_{\lambda_1}p_{\lambda_2}\dotsb p_{\lambda_r} \mid \lambda\vdash n, n\in\N\}$, identifying $p_{(0)}$ with $1$. Let each $p_k$ be primitive so that $\sym$ is Hopf (being the universal envelope of a countably infinite dimensional abelian Lie algebra) and put $C=\coalg(\sym)$. The map $f \colon \nsym\otimes C \to \sym$ given on generators by $f(H_k,p_\lambda) = \delta_{k,|\lambda|}\, p_\lambda$ induces a surjective measuring.
\end{exam}

In the preceding example, one can view the map $f$ as \demph{degree preserving:} taking $\deg (b\otimes c) = \deg c$, we have $\deg f(b\otimes c) = \deg(b\otimes c)$. Such maps will be important in applications of our framework to combinatorics, but we do not insist on this property at present. See Sections \ref{sec: combinat} and \ref{sec: universal}. 

\begin{exam}\label{ex: poly over qsym}
Recall that for any positive integer $n$, we say $\gamma$ is a composition of $n$ (written $\gamma \vDash n$), if $\gamma$ is an ordered tuple of positive integers $(\gamma_1,\ldots,\gamma_r)$ summing to $n$ (written $|\gamma| = \gamma_1+\cdots+\gamma_r = n$). We take $(0)$ to be the unique composition of $0$. The space $\k\{ M_\gamma \mid \gamma\vDash n, n\in\N \}$ becomes an algebra under the shuffle product: $M_\alpha \cdot M_\beta = \sum_{\gamma\in\alpha\shuf\beta} M_\gamma$. See \cite{Ree:58}, \cite{Hof:00}, \cite[Ch. 7]{Sta:99} for details. (Here we identify $M_{(0)}$ with $1$.) This is a presentation of what is known as the ring $\qsym$ of \emph{quasisymmetric functions}. It is a graded connected Hopf algebra with (noncocommutative) deconcatenation coproduct: if $\gamma = (\gamma_1,\gamma_2,\ldots,\gamma_r)$, then $\Delta(M_\gamma) = \sum_{0\leq i\leq r} M_{(\gamma_1,\ldots,\gamma_i)} \otimes M_{(\gamma_{i+1},\ldots, \gamma_r)}$. (Here we take $\gamma_0=\gamma_{r+1}=0$.) The map $f \colon k[z] \otimes \qsym \to \qsym$ given by $f(z^n,M_\gamma) = \delta_{n,r} \, M_{\gamma_1}\cdots M_{\gamma_r}$ is a measuring. 
\end{exam}

The above resembles Example \ref{ex: comm over noncomm}, but there are important differences: (1) here the mapping is a length-graded convolution power (again, useful for applications to combinatorics); and (2) the measuring algebra $B$ is now only a bialgebra, not a Hopf algebra---this distinction will become important in Proposition \ref{th: cocomm}. The above also represents our first example of a family of maps that is not jointly surjective, as $\qsym$ is not generated by the monomial quasisymmetric functions $M_k$. (The proper subalgebra they generate is well-known to be isomorphic to $\sym$.)

%-----------------------------------------------------------------------
\subsection{The bicategory of algebras with measurings}\label{sec: alg}

Here we define a bicategory $\malg_\k$ of algebras with measurings, and collect some of its elementary properties.  

\smallskip\noindent\emph{\underline{Objects:}} algebras over $\k$.

\smallskip\noindent\emph{\underline{$1$-Cells:}} morphisms $f\colon B\to A$ are identified with the set $\Meas(B,A)$ of measurings $f\colon B\ot C_f\to A$. 

\begin{description}
\item[\it Identities] $\id\colon A\to A$ is given by $C_{\id}=\k$.  

\smallskip
\item[\it Composition] for $f\colon B\to A$ and $g\colon D\to B$ we define their composite $h=fg\colon D\to A$ by $C_h=C_g\otimes C_f$ and
$h=f\circ (g\ot\id)$.
\end{description}

\smallskip\noindent\emph{\underline{$2$-Cells:}} if $f,g\colon B\to A$ are $1$-cells, then a 2-cell $\mathbf t\colon f\to g$ is a coalgebra map $\mathbf t\colon C_f\to C_g$ making the following diagram commute.

\begin{center}
\begin{tikzpicture}[]
\matrix (m) [matrix of math nodes, row sep=2.5em,column sep=2.5em]
{ B\otimes C_f &  \\
B\otimes C_{g} & A \\ };
\path[-]%,font=\scriptsize
	(m-1-1) 	edge[->] node[left]{$ \id\otimes\mathbf t $} (m-2-1)
			edge[->] node[above right]{$ f $}  (m-2-2)
	(m-2-1) 	edge[->] node[below] {$ g $} (m-2-2);
\end{tikzpicture} 
%\begin{CD} B\otimes C_f @> f >> A \\
%@V \id\ot t VV  ||\\
%B\otimes C_g @> g>> A
%\end{CD}
\end{center}

\begin{description}
\item[\it Identities] if $f\colon B\to A$ is a $1$-cell, then $\id\colon f\to f$ is the identity map on $C_f$.

\smallskip
\item[\it Horizontal composition] for $f,f'\in\Meas(B,A)$ and $g,g'\in\Meas(D,B)$, and $2$-cells $\mathbf s\colon f\to f'$ and $\mathbf t\colon g\to g'$, we define $\mathbf s\circ \mathbf t\colon fg\to f'g'$ so as to make the diagram at left below commute.
% $s\circ t=s\otimes t\colon C_g\otimes C_f\to C_{g'}\otimes C_{f'}$.

\smallskip
\item[\it Vertical composition] for $f,g,h\in\Meas(B,A)$ and $2$-cells $\mathbf t\colon f\to g$, $\mathbf s\colon g\to h$, we define $\mathbf{st}\colon f\to h$  so as to make the diagram at right below commute.
% $st=c\circ t\colon C_f\to C_h$.

\begin{center}
%% horizontal composition
\begin{tikzpicture}[baseline]
\matrix (m) [matrix of math nodes, row sep=3.5em,column sep=3em]
{%
D\otimes C_f \otimes C_g & B\otimes C_g &  \\
D\otimes C_{f'} \otimes C_{g'} & B\otimes C_{g'} & A \\ 
};
\path[-]%,font=\scriptsize]
	(m-1-1) edge[->] node[left]{$ \id\otimes\mathbf s $} (m-2-1)
		edge[->] node[above right]{$ f $}  (m-1-2)
	(m-1-2) edge[->] node[left]{$ \id\otimes\mathbf t $} (m-2-2)
		edge[->] node[above right]{$ g $}  (m-2-3)
	(m-2-1) edge[->] node[below] {$ f' $} (m-2-2)
	(m-2-2) edge[->] node[below] {$ g' $} (m-2-3);
\end{tikzpicture}
\quad\quad\quad
%
%% vertical composition
\begin{tikzpicture}[baseline]
\matrix (m) [matrix of math nodes, row sep=2.5em,column sep=3.25em]
{B\otimes C_f &  \\
B\otimes C_{g} &  A \\
B\otimes C_{h} & \\
};
\path[-]%,font=\scriptsize]
	(m-1-1) edge[->] node[left]{$ \id\otimes\mathbf t $} (m-2-1)
		edge[->] node[above right]{$ f $}  (m-2-2)
	(m-2-1) edge[->] node[left]{$ \id\otimes\mathbf s $} (m-3-1)
		edge[->] node[below] {$ g $} (m-2-2)
	(m-3-1) edge[->] node[below] {$ h $} (m-2-2);
\end{tikzpicture}
\end{center}
\end{description}

The $1$-cell category $\Meas(B,A)$ inherits many of the properties of the category of coalgebras. 
In particular, pushouts and direct sums (categorical coproducts) exist. In fact, this holds more generally for any slice category of a category with colimits, cf. \cite[Ch. V]{Mac:98}. As the theme of this work is computation, we show how to construct these objects explicitly for related categories in Section \ref{sec: bialg}.

%------------------------------------------------
% Begin SECTION
%------------------------------------------------
\section{Partial Coverings and Transfer of Structure}\label{sec: coverings} 

If $A,B$ are bialgebras and $C$ a coalgebra, then we say that a map $f\colon B\otimes C\to A$ is a \demph{partial covering} if it is a measuring as well as a coalgebra map.  We say that $f$ is a \demph{covering} if it is also surjective.  In what follows, when we wish to deemphasize the indexing coalgebra, we write $B \pcovers A$ to indicate that $B$ partially covers $A$.

%------------------------------------------------
\subsection{The bicategories of bialgebras with coverings and partial coverings}\label{sec: bialg}

We define bicategories $\pcbialg$ and $\cbialg$ of bialgebras with partial coverings and bialgebras with coverings analogous to the bicategory of algebras with measurings. 

\begin{rema}
The bicategories $\malg$, $\pcbialg$, and $\cbialg$ can also be viewed as double categories with horizontal morphisms being (algebra or bialgebra) homomorphisms and vertical homomorphisms being measurings or (partial) coverings. 
% with horizontal morphisms being surjective bialgebra homomorphisms and vertical homomorphisms being coverings.  
%This double category structure will hopefully be explored in future.
We do not explore this viewpoint any further here.
\end{rema}

\medskip\noindent\emph{\underline{Direct Sums:}} Let $B$, $A$ be bialgebras.  
If $f \colon B\otimes C_f\to B$, $g\colon B\otimes C_g\to A$ are (partial) coverings, define the measuring $(h,C_h)$ by putting $C_h:=C_f\oplus C_g$ and $h:=f\circ \iota_f+g\circ\iota_g$.  (Here, $\iota_f, \iota_g$ are the canonical injections of $C_f, C_g$ into $C_f\oplus C_g$.)
One checks that $h$ is a partial covering. If either $f$ or $g$ is a covering, then so is $h$.

\medskip\noindent\emph{\underline{Pushouts:}}  Let $B$, $A$ be bialgebras, and consider coverings $f,g,h\in\PCov(B,A)$ and morphisms $\mathbf s\colon f\to g$, $\mathbf{t}\colon f\to h$ between them. The pushout (fibered coproduct) is the partial covering $(k,C_k)$ completing the diagram
%$$
%\begin{CD}
%f @> t >> g \\
%@V s VV & \\
%h
%\end{CD}
%$$
%is given by 
%$$
%\begin{CD}
%f @> t >> g \\
%@V s VV  @VV \tilde{s} V\\
%h @> \tilde{t} >> k
%\end{CD}
%$$

\begin{center}
\begin{tikzpicture}[baseline]
\matrix (m) [matrix of math nodes, row sep=1.5em,column sep=2em]
{%
& g &  \\
f & & k \\
& {h} & \\
};
\path[-]%,font=\scriptsize]
	(m-1-2) edge[<-] node[xshift=-.5em, yshift=.35em]{$\mathbf s$} (m-2-1)
		edge[->,dashed] node[xshift=.55em, yshift=.5em]{$\mathbf{\tilde{t}}$}  (m-2-3)
	(m-3-2) edge[<-] node[xshift=-.6em, yshift=-.4em]{$\mathbf{t}$} (m-2-1)
		edge[->,dashed] node[xshift=.5em, yshift=-.4em]{$\mathbf{\tilde{s}}$} (m-2-3);
\end{tikzpicture},
\end{center}
where $C_k=C_g\oplus_{C_f\!} C_{h} := (C_g\oplus C_{h})/I$, where $I=\{(\mathbf s(x),-\mathbf{t}(x)) \mid x\in C_f\}$ is the pushout in the category of coalgebras. The partial covering $k\colon B\otimes C_k\to A$ is given by $k(b\otimes ((x,y)+I))=g(b\otimes x)+h(b\otimes y)$.  If $f,g,h$ are coverings, then so is $k$.  If $\mathbf s$ is surjective, then so is $\mathbf{\tilde{s}}$.  Similarly for $\mathbf t,\mathbf{\tilde{t}}$.

\begin{xam}[\ref{ex: A.id over A}]\label{ex: morphism family} 
Suppose $\Xi$ is a jointly surjective family of bialgebra maps from $B$ to $A$ and $C=k\Xi$ is the free pointed coalgebra on $\Xi$. Define a bilinear mapping $f \colon B\ot C\to A$ by $f(b,\xi)=\xi(b)$, for $b\in B$, $\xi\in \Xi$. It is straightforward to check that $f$ is a covering. (In particular, bialgebra maps are coverings.)
\end{xam}

\begin{xam}[\ref{ex: comm over noncomm}]
Let $A$ be a group algebra $\k G$. It is straightforward to check that the measuring $f$ of Example \ref{ex: comm over noncomm} is a coalgebra map.
\end{xam}

\begin{xam}[\ref{ex: nsym over sym}]
The algebra $\nsym=\k\langle H_1,H_2,\ldots\rangle$ becomes a Hopf algebra with coproduct on generators given by $\Delta H_k = \sum_{i+j=k} H_i \ot H_k$. (In fact, as shown in \cite{GKLLRT:95}, it is the graded dual of the Hopf algebra $\qsym$.) We check that the measuring $f$ of Example \ref{ex: nsym over sym} is a coalgebra map: 
\begin{align*}
	\left(f^{\otimes2} \Delta\right)\,(H_k\otimes p_\lambda) &= f^{\otimes2} \sum_{\substack{i+j=k\\ \mu\sqcup\tau=\lambda}} \left(H_i\otimes p_\mu \right) \otimes \left( H_j \otimes p_\tau\right) \\
	&=  \sum_{\substack{i+j=k\\ \mu\sqcup\tau=\lambda}} \delta_{i,|\mu|} \, p_\mu \otimes \delta_{j,|\tau|} \, p_\tau \\
	&= \delta_{k,|\lambda|} \sum_{\mu\sqcup\tau=\lambda} p_\mu \otimes p_\tau \ 
	= \left(\Delta  f\right) (H_k \otimes p_\lambda) .
\end{align*}
\end{xam}

\begin{exam} \label{ex: poly over sym}
We give another covering of $\sym$. Give $\k[x]$ the structure of Hopf algebra by letting $x$ be primitive. Let $C = \coalg(\sym)$, using the power sum basis $\{ p_\lambda\}$ as above. For a partition $\lambda$ of length $r$, put $f(x^n,p_\lambda) = \delta_{n,r} \, {n!}\cdot p_{\lambda}$. It is easy to check that $f \colon \k[x] \ot C \to \sym$ is a measuring. To show that $f$ is a partial covering, we verify that it is a coalgebra map.
\begin{align*}
   \left(f^{\otimes2} \Delta\right)\,(x^n\otimes p_\lambda) 
	&= f^{\otimes2} \! \sum_{\substack{i+j=n\\ I\sqcup J= [r]}} \! \binom{n}{i} \! \left(x^i \otimes p_{\lambda_I}  \right) \otimes \left( x^{j} \otimes p_{\lambda_J} \right) \\
	&= \sum_{\substack{i+j=n\\ I\sqcup J = [r]}} \! \binom{n}{i} \delta_{i,|I|} \, \delta_{j,|J|} \, {i!}{j!} \cdot p_{\lambda_I} \otimes p_{\lambda_J} \\
	&= \delta_{n,r} \, {n!} \sum_{I\sqcup J = [r]} p_{\lambda_I} \otimes p_{\lambda_J} %\,, 
\ = \ 
%\intertext{while}
   \left(\Delta  f\right) (x^n \otimes p_\lambda)	
%	&=  \Delta\left(\delta_{n,r}\,{n!}\cdot p_{\lambda}\right) \\
%	&= \delta_{n,r} \, {n!} \sum_{I \sqcup J = [r]} p_{\lambda_I} \otimes p_{\lambda_J} 
\,.
\end{align*}
\end{exam}

%------------------------------------------------
\subsection{Universal partial coverings}\label{sec: universal coverings}

If $A$ and $B$ are algebras then Sweedler \cite{Swe:69} constructs a universal measuring coalgebra $\mathcal M(B,A)$ and a universal measuring $\Omega\colon B\otimes \mathcal M(B,A)\to A$ satisfying the following universal property: 
\begin{center}
\begin{tikzpicture}[]
\matrix (m) [matrix of math nodes, row sep=3.5em,column sep=3.5em]
{% 
B\otimes C_f &  \\
B\otimes \mathcal M(B,A) & A \\ 
};
\path[-]%,font=\scriptsize]
	(m-1-1) edge[->,dashed] node[left]{$\exists \id\otimes F $} (m-2-1)
		edge[->] node[above right]{$ \forall f $}  (m-2-2)
	(m-2-1) edge[->] node[below] {\small$ \Omega $} (m-2-2);
\end{tikzpicture}
\end{center}
That is, if $f\colon B\otimes C\to A$ is a measuring, then there exists a unique coalgebra map $F\colon C\to \mathcal M(B,A)$ such that $f=\Omega\circ(\id\ot F)$.  %(See Section \ref{sec: universal coverings} for the construction in an adapted setting.) 

This universal object may be viewed as a generalization of Sweedler's finite dual $B^\circ$ (taking $A=\k$). 
In the recent literature, one finds other generalizations \cite{HFV:17, PR:16}, the former of which shares some overlap with our work (see the remark following Proposition \ref{th: universal-bialg}). In Theorem \ref{th: universal-coalg}, we adapt Sweedler's proof to construct universal partial coverings. We remark that the universal object $\C(B,A)$ created there is naturally a subcoalgebra of Sweedler's $\mathcal M(B,A)$.

\begin{theo}\label{th: universal-coalg} Let $A$ and $B$ be bialgebras.  Then there exists a universal covering coalgebra $\C(B,A)$ and a universal partial covering $\Omega\colon B\otimes \C(B,A)\to A$ with the following universal property: if $f\colon B\otimes C\to A$ is any partial covering, then there exits a unique coalgebra map $F\colon C\to \C(B,A)$ such that $F=\Omega\circ(\id\ot F)$.
\end{theo}

\begin{proof}
Let $C$ denote the cofree coalgebra on the space $L(A,B)$ of linear maps from $B$ to $A$ and let $p\colon C \to L(B,A)$ be the canonical projection.  We define a linear map $G\colon B\otimes C\to A$ by $G=\mathrm{ev}\circ (\id\ot p)$, that is, $G(b\otimes c)=p(c)(b)$.  Now we define $\C(B,A)$ to be the largest subcoalgebra of $C$ such that $G|_{B\otimes \C(B,A)}$ is a partial covering and we define $\Omega=G|_{B\otimes \C(B,A)}$.  The existence of such a coalgebra follows from 
the following observation: if $(X_i)_{i\in I}$ is a family of subcoalgebras of $C$ such that $G|_{B\otimes X_i}$ are partial coverings, then for $X=\sum_i X_i$, the map $G|_{B\otimes X}$ is also a partial covering (as being a partial covering 
is described in terms of homogeneous linear equations).  Now suppose that $f\colon B\ot C_f\to A$ is a partial covering.  Let $q\colon C_f\to L(B,A)$ be the corresponding map, that is, $q(c)(b)=f(b,c)$.  By the cofreeness of $p\colon C\to L(B,A)$, there is a unique coalgebra map $F\colon C_f\to C$ such that $q=p\circ F$.  Now note that $G|_{B\otimes F(C_f)}$ is a partial covering and hence $F$ can be viewed as a map from $C_f$ to $\C(B,A)$.  Uniqueness of $F$ follows from the construction (coalgebra maps $F$ from $C_f$ to $\C(B,A)$ satisfying $f=\Omega\circ(\id\ot F)$ correspond injectively to coalgebra maps from  $C_f$ to $C$ satisfying $q=p\circ F$).
\end{proof}

\begin{rema}
In \cite{GM:06, GM:07}, Grunenfelder and Mastnak considered a generalization of measurings to ``bimeasurings'' in the case where $B$ is a bialgebra and the algebra $A$ is commutative.  Among other things they show that the universal measuring coalgebra $\mathcal M(B,A)$ carries a natural structure of a bialgebra. Their argument readily carries over to the present setting of partial coverings. We sketch the proof here.% and refer the reader to \cite{GM:06} for details.
\end{rema}

Suppose $B$ is cocommutative and $A$ is commutative. Define a partial covering $\omega_2\colon B\otimes \C(B,A)\otimes \C(B,A)\to A$ by $\omega_2 = \m_A(\omega\ot\omega)(\id\ot\tau\ot\id)(\De_B\ot \id\ot\id)$ (here $\tau\colon B\ot \C(B,A)\to \C(B,A)\ot A$ denotes the usual twist), that is, $\omega_2(b,x\otimes y)=\Omega(b_1,x)\Omega(b_2,y)$.  Check that $\omega_2$ is a partial covering (cocommutativity of $B$ is needed to ensure that $\De_B$ is a coalgebra map; commutativity of $A$ is needed to ensure that $m_A$ is an algebra map; conclude $\omega_2$ is a measuring). Now define $\m\colon \C(B,A)\otimes \C(B,A)\to \C(B,A)$ to be the unique coalgebra map such that $\omega_2=\Omega(\id\ot\m)$.  The associativity of $\m$ follows from the fact that $\m(\id\ot\m)$ and $\m(\m\ot\id)$ are unique coalgebra maps such that $\omega_3 = \Omega(\id\ot \m(\id\ot\m))$, where $\omega_3\colon B\ot \C(B,A)\ot \C(B,A)\ot \C(B,A)\to A$ is given by $\omega_3(b,x\ot y\ot z)=\Omega(b_1,x)\Omega(b_2,y)\Omega(b_3,z)$.  The unit $u\colon \k\to \C(B,A)$ is the unique coalgebra map such that $\omega_0 = \Omega(\id\ot u)$, where $\omega_0=\ep\colon B\ot \k\to A$. 
This bialgebra structure on $\C(B,A)$ is universal in the following sense.

\begin{prop}\label{th: universal-bialg}
Let $A$ and $B$ be bialgebras with $A$ commutative and $B$ cocommutative. Then the universal covering coalgebra $\C(B,A)$ is a bialgebra satisfying the following universal property: if $C$ is a bialgebra and $f\colon B\ot C\to A$ is a partial bicovering (that is, it may be viewed as either a partial covering of $A$ by $B$ or a partial covering of $A$ by $C$), then there is a unique bialgebra map $F\colon C\to \C(B,A)$ such that $f=\Omega\circ(\id\ot F)$.
\qed
\end{prop}

\begin{rema}
The article \cite{HFV:17} may be relevant here.
In the case that $A,B$ are cocommutative, it seems our $\C(B,A)$ reduces to the authors' \emph{universal measuring comonoid in the category of coalgebras $P(B,A)$.} Theorem 10.11 {\it loc. cit.} provides that if $A$ is furthermore commutative and $B$ is Hopf, then $\C(B,A)$ is also Hopf. We note, in contrast, that $P(B,A)$ is always cocommutative, while that is not the case for $\C(B,A)$; cf. Example \ref{ex: cocomm-cover}.
\end{rema}

\begin{prop} Let $A,B$ be bialgebras.  Then there exits a covering $F\colon B\covers A$ if and only if the universal partial covering $\Omega\colon B\ot \C(B,A) \to A$ is surjective. \qed
\end{prop}

%------------------------------------------------
\subsection{Covering equivalence}

We define two $1$-cells in $\Cov(B,A)$ to be \demph{covering equivalent} if they factor surjectively through another $1$-cell, {\it i.e.}, $f\sim g$ if there exists a covering $h\in \Cov(B,A)$ and surjective morphisms of coverings $\mathbf t\colon f\to h$ and $\mathbf s\colon g\to h$. If instead $f,g\in\PCov(B,A)$, we further stipulate that the ranges of $f$, $g$, and $h$ coincide. 

This is indeed an equivalence relation. Symmetry is obvious. Transitivity is seen by using the pushout construction: if $f_1,f_2$ factor surjectively through $g$ and $f_2, f_3$ factor surjectively through $h$, then $f_1, f_3$ factor surjectively through the pushout $g\oplus_{f_2} h$.

\begin{rema}
In Section \ref{sec: universal coverings} we show that given bialgebras $B,A$, there exists a universal partial-covering $\Omega\colon B\otimes \C(B,A)\to A$ that satisfies the property that any other partial covering factors uniquely through $\Omega$.  It is easy to see that two partial coverings $f,g\in \PCov(B,A)$ are equivalent if and only if the images of their coalgebras $C_f, C_g$ in $\C(B,A)$ coincide.  This observation also gives an alternative prove that the equivalence of coverings above is an equivalence relation.
\end{rema}

\begin{exam} Here we identify two partial coverings that are equivalent but not isomorphic.  Let $\varphi\colon B\to A$ be a bialgebra homomorphism. Let coalgebras $C, D$ be given by $C=\k x\oplus \k y$, $D=\k z$, where $x,y,z$ are points. Define partial coverings (they are coverings if $\varphi$ is surjective) $f\colon B\otimes D\to A$, $g\colon B\otimes C\to A$ by $f(b\otimes x)=f(b\otimes y)=g(b\otimes z)=\varphi(b)$. Then $f$ factors surjectively through $g$ via the coalgebra map $\mathbf t\colon C\to D$ given by $\mathbf t(x)=\mathbf t(y)=z$.
\end{exam}

Let us call a partial covering $f$ \demph{non-degenerate} if the linear map $f(\arg,x)\colon B\to A$
is nonzero for every nonzero $x\in C_f$. Observe that if $f$ is non-degenerate and factors surjectively through $h$, then $f$ and $h$ are isomorphic. Hence two non-degenerate coverings $f,g$ are equivalent if and only if they are isomorphic.

\begin{ques}
Is everything partial covering equivalent to a non-degenerate one? 
\end{ques}

%------------------------------------------------
\subsection{Transfer of structure}
We discuss some useful applications of coverings.

\begin{theo}\label{th: transfer} 
Let ${f} \colon B\ot C \to A$ be a partial covering, with image $A'$. Let $\iota \colon A'\to B\ot C$ be any linear section of ${f}$, that is, ${f}\circ \iota = \id$.
Then the following hold.
\begin{enumerate}
\item If $p\in B$ is primitive then for any $c\in C$, ${f}(p,c)$ is primitive in $A'$.
\item If $A'$ and $B$ are Hopf algebras, then their antipodes are related by the formula $S_{A'} = {f}\circ (S_B\ot\id)\circ \iota$.
\end{enumerate}
\end{theo}

\begin{proof} The proof of (1) is a simple calculation:
\begin{align*}
\Delta {f}(p\ot c) &= ({f}\ot {f})\Delta (p\ot c) \\
&= \sum_{(c)}{f}(p\ot c_1)\ot {f}(1\ot c_2) + \sum_{(c)}{f}(1\ot c_1)\ot {f}(p\ot c_2) \\
&= \sum_{(c)}{f}(p\ot c_1)\ot \ep(c_2) + \sum_{(c)}\ep(c_1)\ot {f}(p\ot c_2) \\
&= {f}(p\ot c)\ot 1 + 1\ot {f}(p\ot c).
\end{align*}

In the proof of (2), we assume without loss of generality that $A'=A$. 

We first show that $S_A {f} = {f}(S_B\ot\id)$. We do this by showing that each is the convolution inverse of ${f}$ in $\mathop{\mathrm{Hom}}_k(B\otimes C, A)$.

Suppose ${f}(x)=a$. Then
   \begin{align*}
   (S_A{f} * {f})(x) &= \sum_{(x)}S_A({f}(x_1)) \,{f}(x_2)  = \sum_{({f}(x))} S_A({f}(x)_1) \, {f}(x)_2 = \sum_{(a)} S_A(a_1) \, a_2 \\
   &= \ep (a) = \ep(x) \,,
\end{align*}
since ${f}$ is a coalgebra map. Similarly, for $b\otimes c \in B\otimes C$, we have
\begin{align*}
   ({f}(S_B\ot\id)*{f})(b\ot c) &= \sum_{(b),(c)} {f}(S_B(b_1)\ot c_1){f}(b_2\ot c_2) = \sum_{(b)} {f}(S_B(b_1)b_2\ot c) \,, \\
   &= {f}(\ep(b)\ot c) = \ep(b) {f}(1\ot c) = \ep(b)\ep(c) \\
   & = \ep(b\ot c) \,,
\end{align*}
where the last equalities on the first and second line use the fact that ${f}$ measures $B$ to $A$. The proof that $S_A {f}$ and ${f}(S_B\ot\id)$ are both right convolution inverses of ${f}$ is similar.  

We conclude that ${f}$ is convolution invertible and that $S_A {f} = {f}(S_B\ot\id)$. This implies that ${f}(S_B\otimes\id)\iota = S_A{f}\iota = S_A$, as needed. In particular, ${f}(S_B\otimes\id)\iota$ is well-defined and independent of the choice of $\iota$.
\end{proof}

%\Aaron{If $b\in B$ is not primitive and $f(b,c) \neq0$, is it true that $f(b,c)$ is not primitive? Stated differently: $f$ is surjective, so we hit all primitives; do we hit all of them with primitives? 
%\\[2ex]
%Probably not: consider some Hopf algebra with skew-primitives, and take the (Hopf?) quotient induced by $G\onto \{1\}$. This gives a covering as in Example \ref{ex: A.id over A}. However, if $A,B$ are cocommutative, then maybe we can argue that $f:B\to A$ is ``$\log(\id - \ep)$-equivariant.''}

Let ${f}\colon B\ot C\to A$ be a covering.  Let $P(B)$ and $P(A)$ denote the subspaces of primitive elements in $B$ and $A$.   One of the important aspects of coverings is that ${f}(P(B)\ot C)\subseteq P(A)$ (just proven above) and hence also ${f}(\mathcal{U}P(B)\ot C)\subseteq \mathcal{U}P(A)$ (here $\mathcal{U}P(B)$ stands for the Hopf subalgebra of $B$ generated by $P(B)$; in characteristic $0$ it is isomorphic to the universal envelope of the Lie algebra $P(B)$).  In a sense this means that ${f}$ preserves the Lie-theoretic aspects of $B$.  

\begin{rema} In contrast to the preceding paragraph, coverings $f\colon B\covers A$ need not respect the group-theoretic aspects of $B$. That is, if $g$ is a grouplike element in $B$, then  ${f}(g,\arg)$ may not even land in the span of the grouplikes of $A$. See Example \ref{ex: comm over noncomm} for one such example.
\end{rema}

We now argue that coverings should be viewed as moving between equations in convolution algebras. Fix a measuring $f \colon B \ot C \to A$. Then any algebra map $\chi\in \Alg(A,\k)$ gives rise to an element ${}^{f\!}\chi \in\Alg(B,C^*)$. Indeed, putting $\langle {}^{f\!}\chi(b),c\rangle := (\chi f)(b\ot c)$, we have 
\begin{align*}
	\langle {}^{f\!}\chi(bb'),c\rangle 
	&= (\chi f)(bb'\otimes c) = \chi\sum f(b,c_1)\,f(b',c_2) \\
	&= \sum \chi f(b,c_1)\, \chi f(b',c_2) 
	= \langle {}^{f\!}\chi(b),c_1\rangle \langle {}^{f\!}\chi(b'),c_2\rangle \\
	&= \langle {}^{f\!}\chi(b) \cdot {}^{f\!}\chi (b'), c\rangle .
\end{align*}
Now, if $A$ and $B$ are bialgebras then $\Alg(A,\k)$ and $\Alg(B,C^*)$ are monoids, under the convolution product. So properties of $\chi$ are informed by properties of ${}^{f\!}\chi$, and vice versa. 
We illustrate with some explicit examples.

\begin{exam}
Let $B$ be a filtered Hopf algebra. There is a well-known formula for the antipode in the convolution algebra $\Hom(B,B)$ due to Takeuchi. The fact that primitives may be seen as solutions to a similar formula seems to be less appreciated: if $B$ is connected and cocommutative, then the \emph{first Eulerian idempotent}  $\mathsf{e}_1 := \log(\id - \ep) \in \Hom(B,B)$ is a projection onto $P(B)$. (See \cite[Ch. 4]{Lod:92} for a proof.)
%, and \cite[\S9]{Schm:94} and \cite{Pat:93} for independent work.)
%-------------
%Quick proof in graded connected case: 
%-------------
%We know it for $\nsym$ and coverings push forward the kernels of such power series.
%
%\begin{quote} \it Fix $F \in \k[[x]]$ which is locally finite on $B$ and $A$. E.g., if $B$ and $A$ are filtered and $F = G(id-\varepsilon) \in \Hom(*,*)$. Fix any covering ${f}\colon B\otimes C \to A$. Then $K = \ker F(\id)$ satisfies
%\[
%	{f}(K\otimes C) \subseteq \ker F(id).
%\]
%\end{quote}
%-------------
\end{exam}

\begin{exam}
If ${}^{f\!}\chi$ has order $n$, then the order of $\chi$ divides $n$. Indeed, if $b$ has exponent $n$, then $f(b,c)$ has exponent dividing $n$, for all $c\in C$. 
\end{exam}

%------------------------------------------------
% Begin SECTION
%------------------------------------------------
\section{Cocommutative and Graded Coverings}\label{sec: subcategories} 

We investigate subcategories of $\pcbialg$.

%------------------------------------------------
\subsection{Cocommutative coverings}\label{sec: cocomm}

We now remark that partial coverings $f\colon B\ot C \to A$ with $C$ cocommutative are less restrictive than those with $C$ non-cocommutative. 
Indeed, by the measuring and coalgebra-map properties, we have
\begin{align}
\notag	\left(f^{\otimes2} \Delta\right)\,(bb'\ot c) 
	&= f^{\otimes2} \left[  \left(b_1b'_1 \ot c_1 \right) \otimes \left(b_2b'_2 \ot c_2\right) \right] \\
\label{eq: Delta-f}
	&= f(b_1,c_1)\,f(b'_1,c_2) \ot f(b_2,c_3)\,f(b'_2,c_4) ,
\intertext{while}
\label{eq: f-Delta}
	\left(\Delta f\right)(bb'\ot c) 
	&= f(b_1,c_1)\,f(b'_1,c_3) \ot f(b_2,c_2)\,f(b'_2,c_4) .
\end{align}
Note the twist $(c_2,c_3) \leadsto (c_3,c_2)$ above, so that $C$ behaves under $f$ somewhat as if it were cocommutative. If $B$ is Hopf, then it behaves precisely so, as we now show.
\begin{lemm}\label{lem: Ccoc}
If $f\colon B\otimes C\to A$ is a partial covering with $B$ Hopf, then for all $b,b'\in B$ and $c\in C$ we have that
$$
f(b',c_2)\otimes f(b, c_1)= f(b', c_1)\otimes f(b,c_2).
$$
\end{lemm}
\begin{proof} In the equality between \eqref{eq: Delta-f} and \eqref{eq: f-Delta}, the antipode offers a means to cancel the edges. Using this equality between the fourth and fifth line below, we compute:
\begin{eqnarray*}
f(b',c_1)\otimes f(b,c_2) &=& \ep(b_1\ot c_1)f(b'_1, c_2)\ot f(b_2,c_3)\ep(b'_2\ot c_4) \\
&=& f(\ep(b_1), c_1)f(b'_1,c_2)\ot f(b_2,c_3)f(\ep(b'_2), c_4) \\
&=& f(S(b_1)b_2, c_1)f(b'_1,c_2)\ot f(b_3,c_3)f(b'_2S(b'_3), c_4)\\
&=& f(S(b_1), c_1)f(b_2, c_2)f(b'_1,c_3)\ot f(b_1,c_4)f(b'_2, c_5)f(S(b'_3), c_6)\\
&=& f(S(b_1), c_1)f(b_2, c_2)f(b'_1,c_4)\ot f(b_3,c_3)f(b'_2, c_5)f(S(b'_3), c_6)\\
&=& f(S(b_1)b_2, c_1)f(b'_1,c_3)\ot f(b_3,c_2)f(b'_2S(b'_3), c_4)\\
&=& \ep(b_1)\ep(c_1)f(b'_1,c_3)\ot f(b_2,c_2)\ep(b'_2)\ep(c_4)\\
&=& f(b', c_2)\otimes f(b,c_1).
\end{eqnarray*}
\end{proof}

\begin{prop}\label{th: cocomm}
If $B \pcovers A$ is a partial covering with $B$ Hopf, then the image of any cocommutative element of $B$ is cocommutative in $A$.
\end{prop}
\begin{proof}
Applying Lemma \ref{lem: Ccoc}, we may write $\Delta{f}(b,c)$ as
\begin{eqnarray*}
{f}(b,c)_1\ot {f}(b,c)_2 &=& {f}(b_1,c_1)\ot {f}(b_2,c_2) = f(b_1,c_2)\ot f(b_2,c_1)\\
&=& {f}(b_2,c_2)\ot {f}(b_1,c_1) = {f}(b,c)_2\ot {f}(b,c)_1 ,
\end{eqnarray*}
which completes the proof.
\end{proof}

Below we provide an example showing that the Hopf condition in Proposition \ref{th: cocomm} (and in Lemma \ref{lem: Ccoc}) is not vacuous. 
That is, there are interesting coverings by cocommutative $B$'s and non-cocommutative $C$'s. It therefore follows (from Example \ref{ex: cocomm-cover}) that the universal covering coalgebra need not always be cocommutative.  Whether this statement holds when $B$ is Hopf is an open question.
\begin{ques} \label{q:cocommutative}
Given any partial covering ${f}\colon B p\covers A$, with $B$ a Hopf algebra, does it factor through a partial covering ${f}'\colon B\ot C'\to A$ with $C'$ cocommutative? Equivalently, is the universal covering coalgebra $\mathcal{C}(B,A)$ cocommutative whenever $B$ is Hopf?
\end{ques}
Along the same lines, we may also ask the following.
\begin{ques} 
Given any partial covering ${f}\colon B\pcovers A$, does there exist a cocommutative $C'$ and partial covering ${f}'\colon B\ot C'\to A$ with the same range as ${f}$?  What if we additionally assume that 
$B$ is a Hopf algebra?
\end{ques}

\begin{exam}\label{ex: cocomm-cover}
Let $\{a,b\}$ be the semigroup with $xy = y$ for all $x,y\in\{a,b\}$, and let $M=\{e,a,b\}$ be the corresponding (nonabelian and non-cancelable) monoid. Put $A = (\k M)^*$ and $C=\coalg(A)$. Writing $\{x_i\}_{i=e,a,b}$ for the dual basis in $A$ and putting $\rho:=x_e + x_a + x_b$, one checks that
\[
x_i \cdot x_j = \delta_{ij} x_j \ \ (\rho=1_A), 
	\quad \Delta(x_e) = x_e\otimes x_e ,
	\qand \Delta(x_i) = 1\otimes x_i + x_i \otimes x_e \ (i\neq e).
\]
%and in particular, $\Delta(\rho) = \rho\otimes\rho$. 
We construct a covering $f\colon B\ot C \onto A$.

Give $B:=\k[z]$ the bialgebra structure with $z$ a point, where by \demph{point} we mean a grouplike element. (We adopt this name to avoid the connotation of invertibility.)
Define $f$ by sending $(z^n,c)$ to the Sweedler power $\id^{*n}(c)$ for $n\geq0$ and extending bilinearly. Then $f(z,c) = c$, meaning $f$ is onto. Note that $f$ is a measuring by fiat. We check that $f$ is a coalgebra map. In the present context, \eqref{eq: Delta-f} and \eqref{eq: f-Delta} requires the equality 
\begin{gather}\label{eq: ncc-example}
(c_1c_3c_5\cdots) \otimes (c_2c_4c_{6}\cdots) = (c_1c_2c_3\cdots) \otimes (c_{n+1}c_{n+2}c_{n+3}\cdots) .
\end{gather}
This clearly holds for $c=x_e$, so we turn to $x_i$ for $i\neq e$. 
The general formula for iterated coproducts is
\[
\Delta^n(x_i) = 1^{\ot n}\ot x_i + 1^{\ot n-1}\ot x_i\ot x_e + 1^{\ot n-2}\ot x_i\ot x_e^{\ot2} + 
\cdots x_i \ot x_e^{\ot n}.
\]
Hence the only products of $n$ factors (with or without interleaving the odd and even terms) that survive in \eqref{eq: ncc-example} are $1\ot x_i$ and $x_i\ot x_e$, and these each occur exactly once on either side.
\end{exam}

%------------------------------------------------
\subsection{Graded coverings}\label{sec: combinat}

If $A,B$ are graded (bi)algebras, then a measuring $f\colon B\ot C \to A$ is called \demph{graded} if $f(b\ot c)_n = f(b_n \ot c)$. We say that $f$ is \demph{locally finite} if for each fixed $c\in C$, $f(b,c) = 0$ for all $b$ of sufficiently large degree. 
Finally, suppose $C$ is also graded. Then we call $f$ \demph{bigraded} if $f(b \ot c)_n = f(b_n \ot c_n)$.

\begin{prop}\label{th: locally finite}
Any graded partial covering $f\colon B \ot C \to A$ of graded bialgebras may be replaced by a locally finite one.
\end{prop}

\begin{proof}
Let $\ngens$ denote the coalgebra $\k\{H_0,H_1,H_2,\ldots\}$ defined by $\Delta(H_k) = \sum_{i+j=k} H_i \otimes H_j$. Put $\bar C = C \ot \ngens$, and $\bar f(b, c \ot H_n) = f(b, c)_n$, the $n$th homogeneous component of $f(b,c)$. Show that this is an measuring and a coalgebra map.
\end{proof}

\begin{xam}[\ref{ex: A.id over A}]
Suppose $A$ is a graded bialgebra. Applying the construction $(f,\k) \mapsto (\bar f,\k\ot\ngens)$ in the proof above, we see that $\bar f$ not only becomes locally finite, but it becomes degree-preserving as well. Thus every graded bialgebra has a degree-preserving measuring that is a covering. 
\end{xam}

%\subsubsection{Diagonal graded tensor products}\label{sec: combinat}
If $U,V, W$ are graded vector spaces, then we define the \demph{diagonal tensor product} by $U\otimes_d V := U\otimes V/(\bigoplus_{n\not=m} U_n\otimes V_m)$. The space $U\otimes_d V$ is graded by $(u\otimes_d v)_n = u_n\otimes v_n$.
Note that bigraded maps from $U\otimes V\to W$ are in bijective correspondence with graded (that is, homogeneous of degree $0$) maps from $U\otimes_d V\to W$. 

\begin{lemm}
If $C,D$ are graded coalgebras, then $C\otimes_d D$ is a graded coalgebra with coproduct given by $\Delta(u\otimes_d v)=
(u_1\otimes_d v_1)\otimes (u_2\otimes_d v_2)$ and counit given by $\ep(u\otimes_d v) = \ep(u)\ep(v)$.
\end{lemm}
\begin{proof}  Observe that $\bigoplus_{n\not=m} C_n\otimes D_m$ is a coideal in $C\otimes D$.
\end{proof}

%------------------------------------------------
%\subsubsection{The bicategory of graded connected bialgebras with bigraded coverings}\label{sec: combinat}
The bicategory of graded connected bialgebras with bigraded coverings $\combinat$ is now defined in analogy to the category of bialgebras with coverings:  we replace tensors by diagonal tensors and coverings by bigraded locally-finite coverings.

\begin{rema} All bicategories described above also have monoidal structures given by the tensor products of underlying vector spaces.
\end{rema}

%------------------------------------------------
\subsection{Universal objects in $\combinat$}\label{sec: universal}

Any cocommutative bialgebra $A$ in $\combinat$ has a canonical covering by $\nsym$, defined as follows. For all $a\in A$, put $\can(H_n,a) = a_n$ (the homogeneous component of $a$ of degree $n$), then extend multiplicatively and bilinearly to make $\can$ a measuring. Checking that $\can$ is a coalgebra map is straightforward. Note that this covering is locally finite by construction.

\begin{prop}\label{th: unique factorization}
Given a covering $f\colon \nsym \ot C \to A$, there is a coalgebra map $ \bar f \colon  C \to \coalg(A)$ factoring $f$  through $\can$ if and only if $C$ is a locally finite covering.
\end{prop}

\begin{proof}
Put $\bar f(c) = \sum_{n\geq0} f(H_n,c)$.
\end{proof}

\begin{exam}\label{ex: omp over nsym}
Let $\omp$ (for \demph{ordered multiset partitions}\footnote{\,The reader is cautioned to interpret \emph{ordered multiset partition} as an abbreviation for ``ordered partition of a multiset into sets.'' While the blocks of a multiset partition may have repeated entries, in general, recent combinatorics literature, {\it e.g.}, \cite{HRW:18}, excludes these cases from consideration.}) denote the free algebra on finite nonempty subsets $K\subseteq \N$. It inherits the structure of Hopf algebra if the coproduct on generators $K$ is defined by $\Delta(K) = \sum_{I\sqcup J = K} I \otimes J$. 
Identify $\emptyset$ with the unit in $\omp$. We show that $\omp$ covers $\nsym$, and hence is another weakly initial object in the subcategory of cocommutative objects in $\combinat$. 

We take $C$ to be the coalgebra $\ngens=\k\{H_0,H_1,H_2,\ldots\}$ from the proof of Proposition \ref{th: locally finite}. The  map $f \colon \omp\ot C \to \nsym$ given on generators by $f(K,H_k) = \delta_{|K|,k} H_k$ induces a measuring, where $|K|$ is the cardinality of the set $K$ and $H_0$ is identified with the unit of $\nsym$. 

We check that $f$ is a coalgebra map:
\begin{align*}
	\left(f^{\otimes2} \Delta\right)\,(K\otimes H_k) &= f^{\otimes2} \sum_{\substack{I+J=K\\ i+j=k}} \left(I\otimes H_i \right) \otimes \left( J \otimes H_j\right) \\
	&=  \sum_{\substack{I\sqcup J = K\\i+j=k}} \delta_{|I|,i} \, H_i \otimes \delta_{|J|,j} \, H_j \\
	&= \delta_{|K|,k} \sum_{i+j=k} H_i \otimes H_j \ 
	= \left(\Delta  f\right) (K \otimes H_k) .
\end{align*}
\end{exam}

We have now found two weakly initial objects for the cocommutative part of  $\combinat$. 

\begin{ques} Can we drop ``weakly'' in this bicategory: 
Is there a sensible equivalence relation on $\combinat$ making $\nsym$ and $\omp$ equivalent?
\end{ques}

\begin{ques} Are there non-cocommutative or non-graded analogs of $\nsym$ or $\omp$?
\end{ques}

%------------------------------------------------
\section{A Nichols Theorem for Hopf coverings}\label{sec: Hopf}

Let $B$ be a Hopf algebra. A result of Nichols \cite{Nic:78} gives some conditions under which a bialgebra quotient $B/I$ is in fact Hopf. For example, it is Hopf if $B/I$ is commutative or if $B$ is cocommutative or pointed. We now present two analogous results for coverings. (In light of Example \ref{ex: morphism family}, these are in fact generalizations of his result.)

\begin{theo}\label{th: fin-dim-cover}
If a Hopf algebra $B$ covers a finite dimensional bialgebra $A$, then $A$ is also a Hopf algebra.  
\end{theo}
\begin{proof}
We consider $A$ as a (faithfully flat, left) $A$-comodule algebra, and study the Galois map $\beta\colon A\otimes A\to A\otimes A$ given by 
\[
\beta(a,a')=a_1\otimes a_2a'.
\]
A result of Schauenburg \cite{Sch:97} provides that if $\beta$ is bijective, then $A$ is Hopf.
We prove that $\beta$ is surjective and therefore bijective. 

Let ${f}\colon B\ot C\to A$ be the covering
and let $\gamma\colon B\ot C\ot A\to B\ot C\ot A$ be the linear map given by $$\gamma(b,c,a)=b_1\ot c_1\ot {f}(b_2,c_2)a.$$ Note that $\gamma'\colon B\ot C\ot A\to B\ot C\ot A$ given by $$\gamma'(b,c,a)=b_1\ot c_1\ot {f}(S(b_2),c_2)a$$  is a right composition inverse of $\gamma$, so $\gamma$ is surjective. Indeed 
\begin{eqnarray*}(\gamma\circ \gamma')(b\ot c\ot a)&=& \gamma(b_1\ot c_1\ot {f}(S(b_2),c_2)a) \\
&=& b_1\ot c_1\ot {f}(b_2,c_2){f}(S(b_3),c_3)a =
b_1\ot c_1\ot {f}(b_2S(b_3), c_2)a \\
&=& b_1\ot c_1 \ot {f}(\ep(b_2),c_2)a = b_1\ot c_1\ot \ep(b_2)\ep(c_2)a = b\ot c\ot a.
\end{eqnarray*}

Now note that $\beta({f}\otimes\id_A) = ({f}\otimes\id_A)\gamma$:
\begin{eqnarray*}
\beta({f}\ot\id_A)(b,c,a) &=& \beta({f}(b,c)\ot a) \\
&=&{f}(b,c)_1\ot {f}(b,c)_2 a= {f}(b_1,c_1)\ot {f}(b_2,c_2) a\\
&=&  ({f}\ot\id_A)(b_1\ot c_1\ot{f}(b_2,c_2)a) = ({f}\ot\id_A)\gamma(b,c,a).
\end{eqnarray*}
So since $({f}\otimes\id_A)\gamma$ is surjective, so must be $\beta$. 

\end{proof}

We now consider the case when $A$ is pointed.

\begin{lemm}\label{lem: p}
If $f\colon B\pcovers A$ is a partial cover over an algebraically closed field, and $B$ is a pointed Hopf algebra, then its image is pointed in $A$.
\end{lemm}
\begin{proof}
Any simple subcoalgebra of $A$ is the image of $B'\ot C'$ for some simple subcoalgebras $B',C'$. Hence $B' = \k\{z\}$ for some {grouplike} $z\in B$. Using Proposition \ref{th: cocomm}, we see that $f(B'\ot C')$ is cocommutative. Hence $f(B\ot C)$ is pointed.
\end{proof}

\begin{lemm}\label{lem: phi-exists}
Suppose $f\colon B\ot C \to A$ is a partial covering and $z$ is an invertible point in $B$. Then the map $\p=f(z,\arg)\colon C\to A$ is convolution coalgebra map whose convolution inverse $\pinv = f(z^{-1},\arg)$ is also a coalgebra map.
\end{lemm}

\begin{proof}
A simple check 
\begin{align*}
(\pinv\ast\varphi)(c) &= \pinv(c_1)\p(c_2) \\
	&=f(z^{-1},c_1)f(z,c_2) = f(z^{-1}z,c) = f(1,c) = \ep(c).
\end{align*}
Similarly for $\p\ast\pinv=\ep$.
\end{proof}

\begin{lemm}\label{lem: phi-is-nice}
Assume that $\k$ is algebraically closed, $C$ a simple coalgebra and $A$ a pointed bialgebra.  Assume that $\p\colon C \to A$ is a convolution invertible coalgebra map (in $\Hom(C,A)$) such that its convolution inverse $\pinv$ is a coalgebra map. Then every point in the range of $\phi$ is invertible.
\end{lemm}

\begin{proof}
First note that for all $c\in C$ we have $\p(c_2)\pinv(c_1)=\ep(c)=\pinv(c_2)\p(c_1)$. The later equality is proven as follows:
\begin{eqnarray*}
\ep(c)&=& \pinv(c_1)\ep(c_2)\p(c_3) = \pinv(c_1)(\m\Delta(\ep(c_2)))\p(c_3)\\
&=&\pinv(c_1)(\m\Delta \p(c_2)\pinv(c_3))\p(c_4)\\
&=&\pinv(c_1)(\p(c_2)\pinv(c_4)\p(c_3)\pinv(c_5))\p(c_6)\\
&=&(\pinv(c_1)\p(c_2))\pinv(c_4)\p(c_3)(\pinv(c_5))\p(c_6))\\
&=&\ep(c_1)\pinv(c_3)\p(c_2)\ep(c_4) = \pinv(c_2)\p(c_1);
\end{eqnarray*}
the proof of the former works similarly.

Since $C$ is simple and $\k$ is algebraically closed, we have that $C$ is isomorphic to a matrix coalgebra $\k\{e^{ij}\}_{1\leq i,j\leq n} = \mathcal{M}_n(\k)^*$.  Let $A'=\p(C)^*$.  We identify $A'$ with a subalgebra of $C^*=\mathcal{M}_n(\k)$ (via $\p^*\colon A'\to C^*=\mathcal{M}_n(\k)$).  Let $A'=S\oplus N$ be the Wederburn-Malcev decomposition of $A'$ into its semisimple part $S$ and the radical $N$.  Since $A$ and therefore $\p(C)$ is pointed (and hence $A'$ is a basic algebra) we have that $S$ is isomorphic to a direct product of copies of the ground field.  We additionally assume (using a simultaneous similarity if necessary), that $S$ is a subset of the set of diagonal matrices in $\mathcal{M}_n(\k)$, and $N$ is a subset of the set of all strictly-upper triangular matrices in $\mathcal{M}_n(\k)$ 
%\mitja{This is a indeed a well known fact, and I believe not at all hard - but I do not know a reference - it would probably take me at least a day to think of a reasonably short proof.  This fact is offcourse trivial if A' is commutative, as in this case $S$ is unique (By Wederburn-Malcev it is always unuque up to conjugating with an element of the form $1+n$, where $n\in N$).  Do we want to just assume that $A'$ is commutative ({\it i.e.}, $\p(C)$ cocommutative)?}  

We will now show that every $\p(e^{ii})$ is invertible, with inverse $\pinv(e^{ii})$, from whence the result follows (the above discussion shows that every point in $\p(C)$ is in the span of $\{\p (e^{ii}): i=1\,\ldots, n\}$.  Since the set of all distinct points is linearly independent, we have that every point if $\p(C)$ is equal to some $\p(e^{ii})$).  Since $A'$ is upper triangular we also note that for all $1\le i<j\le n$ we have that $\p(e^{ij})=0$.  We will prove by induction on $i$ that for every $i=1,\ldots, n$, we have $\p(e^{ii})\pinv(e^{ii})=1=\pinv(e^{ii})\p(e^{ii})$ and for all $j>i$, $\pinv(e^{ji})=0$.

\smallskip\noindent
\textbf{Base case.} Writing $\m$ for multiplication in $A$, we first compute
\begin{align*} 1 & = \ep(e^{11}) = \m(\pinv\ot\p)\Delta(e^{11}) 
 = \sum_{k=1}^n\pinv(e^{1k})\p(e^{k1})  =  \pinv(e^{11})\p(e^{11}),
\intertext{and}
1& = \ep(e^{11}) = \m(\p\ot\pinv)\Delta^{cop}(e^{11}) 
 = \sum_{k=1}^n\p(e^{k1})\pinv(e^{1k})  =  \p(e^{11})\pinv(e^{11}).
\intertext{For $j>1$, we further have}
0 & =  \ep(e^{j1})  =  \m(\pinv\ot\p)\Delta(e^{j1}) 
 = \sum_{i=1}^n \pinv(e^{ji})\p(e^{i1})  =  \pinv(e^{j1})\p(e^{11}).
\end{align*}
Since $\p(e^{11})$ is invertible we conclude that $\pinv(e^{j1})=0$.

\smallskip\noindent
\textbf{Induction step.} Assume that for some $m$, $2\le m\le n$ we have that every $i<m$
satisfies $\p(e^{ii})\pinv(e^{ii})=1=\pinv(e^{ii})\p(e^{ii})$ and for all $j>i$, $\pinv(e^{ji})=0$.  Then
\begin{align*} 1& = \ep(e^{mm})=\m(\pinv\ot\p)\Delta(e^{mm})
=\sum_{k=1}^n\pinv(e^{mk})\p(e^{km})\\ 
& =  \sum_{k<m} \pinv(e^{mk})\p(e^{km})+\pinv(e^{mm})\p(e^{mm})+\sum_{k>m} \pinv(e^{mk})\p(e^{km})\\
& =  0 + \pinv(e^{mm})\p(e^{mm}) +0
\intertext{and}
1 & = \ep(e^{mm})=\m(\p\ot\pinv)\Delta^{cop}(e^{mm})
=\sum_{k=1}^n\p(e^{km})\pinv(e^{mk})\\ 
& = \sum_{k<m} \p(e^{km})\pinv(e^{mk})+\pinv(e^{mm})\p(e^{mm})+\sum_{k>m} \p(e^{km})\pinv(e^{mk})\\
& =  0 + \p(e^{mm})\pinv(e^{mm}) +0;
\intertext{and for $j>m$, we have}
0 & =  \ep(e^{jm}) = \m(\pinv\ot\p)\Delta(e^{jm}) \\
& = \sum_{i=1}^n \pinv(e^{ji})\p(e^{im}) = \pinv(e^{j1})\p(e^{11})\\
& =  \sum_{i<m} \pinv(e^{ji})\p(e^{im}) + \pinv(e^{jm})\p(e^{mm}) +
\sum_{i>m} \pinv(e^{ji})\p(e^{im}) \\
& =  0+\pinv(e^{jm})\p(e^{mm})+0.
\end{align*}
Therefore, since $\p(e^{mm})$ is invertible, we get $\pinv(e^{jm})=0$.
\end{proof}
\begin{theo}\label{th: pointed-cover}
Let $f\colon B\covers A$ be a bialgebra covering. If $B$ is a Hopf algebra with cocommutative coradical, then $A$ is a Hopf algebra as well.
\end{theo}

\begin{proof}
With no loss of generality we can and do assume that $\k$ is algebraically closed.  We use Takeuchi's condition: an antipode exists on $A$ if it may be defined on the coradical of $A$.
Looking at that coradical, we use Proposition \ref{th: cocomm} to conclude that $A$ is pointed.  Note that every point in $A$ is in the image, under the covering, of some
$\k\{z\}\ot C\cong C$, where $z$ is a point in $B$ and $C$ is a simple sub-coalgebra of $C_f$.  Hence by Lemmas \ref{lem: phi-exists} and \ref{lem: phi-is-nice} we have that every point in $A$ is invertible. 
\end{proof}

%------------------------------------------------
% Begin BIBLIOGRAPHY
%------------------------------------------------
%%\bibliographystyle{plain}
%%\bibliographystyle{amsalpha}
%%\bibliographystyle{alphanum}
%%\bibliographystyle{alpha}

%\bibliographystyle{abbrv}
%\bibliography{bibl}  

\end{document}